\documentclass[11pt, leqno]{amsart}
\makeatletter  
\usepackage{amsfonts,delarray,amssymb,amsmath,amsthm,a4,a4wide}
\usepackage{color}

\theoremstyle{plain}
\newtheorem{thm}{Theorem}[section]
\newtheorem{corollary}[thm]{Corollary}
\newtheorem{lemma}[thm]{Lemma}
\newtheorem{propo}[thm]{Proposition}

\newtheorem{rem}[thm]{Remark}

\numberwithin{equation}{section}
\usepackage[margin=1in]{geometry}

\newcommand{\dist}{\text{dist}} %distance

\newcommand{\C}{\mathbb C}
\newcommand{\R}{\mathbb R}
\newcommand{\be}{\begin{equation}}
\newcommand{\ee}{\end{equation}}

\newcommand{\ep}{\eps}

\renewcommand{\H}{\mathbb H}
\newcommand{\eps}{\varepsilon}

\newcommand{\ov}{\overline}
\newcommand{\Om}{\Omega}
\newcommand{\p}{\partial}

\newcommand{\comment}[1]{}

\newenvironment{myindentpar}[1]%
{\begin{list}{}%
         {\setlength{\leftmargin}{#1}}%
         \item[]%
}
{\end{list}}

\title[The Monge-Amp\`ere equation on polygonal domains]{
Global $C^{2,\alpha}$ estimates for the Monge-Amp\`ere equation on polygonal domains  in the plane}
\author{Nam Q. Le}
\address{Department of Mathematics, Indiana University, Bloomington, IN 47405, USA}
\email{nqle@indiana.edu}
\author{Ovidiu Savin}
\address{Department of Mathematics, Columbia University, New York, NY 10027, USA}
\email{savin@math.columbia.edu}
\thanks{ N.Q.L. was supported by NSF grant DMS-1764248. O.S. was supported by NSF grant DMS-1500438.}
\begin{document}

\begin{abstract}

We classify global solutions of the Monge-Amp\`ere equation $\det D^2 u=1 $ on the first quadrant in the plane with quadratic boundary data. As an application, we obtain global $C^{2,\alpha}$ estimates for the  non-degenerate Monge-Amp\`ere equation in convex polygonal domains in $\mathbb R^2$ provided a globally $C^2$, convex strict subsolution exists.

\end{abstract}
\maketitle

\section{Introduction and statement of the main results}
In this paper, we establish global
$C^{2,\alpha}$ estimates for the non-degenerate Monge-Amp\`ere equation in convex polygonal domains in $\mathbb R^2$ provided a globally $C^2$, convex strict subsolution exists.

For smooth domains $\Omega$ in $\R^n$, boundary $C^2$ estimates for the convex solution to the Dirichlet problem for the Monge-Amp\`ere equation
$$ \det D^2 u = f\quad \text{ in} ~  \Omega, \quad
u = \varphi\quad \text{ on}~ \p\Omega$$
in the nondegenerate case where $f\in C(\ov \Om)$ and $f>0$ in $\ov\Om$, have received considerable attention in the last four decades.
   On smooth and strictly convex domains $\Omega$, these boundary estimates
 were obtained starting with 
the works of Ivo\u{c}kina \cite{I}, Krylov \cite{K}, Caffarelli-Nirenberg-Spruck \cite{CNS} (see also Wang \cite{W}). 
Also on convex domains, global $C^{2,\alpha}$ estimates under sharp conditions
on the right hand side and boundary data were obtained by Trudinger-Wang \cite{TW1} and the second author \cite{S1}. 
On bounded smooth domains $\Omega$ that are not necessarily convex,   global $C^{2,\alpha}$ estimates with globally smooth right hand side and boundary data were 
first obtained by Guan-Spruck \cite{GS} under the assumption that there exists a convex strict subsolution $\underbar{u}\in C^2(\ov\Om)$ taking the boundary values $\varphi$. 
The strictness of the subsolution $\underbar{u}$ in \cite{GS} was later removed by Guan \cite{G}.

In this paper, we relax the smoothness of the domains $\Omega$ in the two dimensional case and investigate $C^{2,\alpha}$ estimates in general convex domains with corners. 

Our first main result states:

\begin{thm}\label{T1}
Let $\Omega$ be a bounded convex polygonal domain in $ \mathbb R^2$.
Let $u$ be a convex function that solves the Dirichlet problem for the Monge-Amp\`ere equation 
\begin{equation}
\label{DP}
\left\{
 \begin{alignedat}{2}
   \det D^2 u ~& = f~&&\text{in} ~  \Omega, \\\
u &= \varphi~&&\text{on}~ \p\Omega.
 \end{alignedat} 
  \right.
  \end{equation}
  Assume that for some $\beta\in (0, 1)$, $$f \in C^\beta(\overline{\Om}), \quad f>0, \quad{ and} \quad \varphi \in C^{2,\beta}(\p \Om),$$ and there is a globally $C^2$, convex, strict subsolution $\underbar{u} \in C^2(\ov \Om)$ to \eqref{DP} (that is, $\det D^2\underbar u>f$ in $\ov\Om$ and $\underbar{u}=\varphi$ on $\p \Om$). Then
$$u \in C^{2,\alpha}(\ov \Om),$$
for some $\alpha>0$.  The constant $\alpha$ and the global $C^{2,\alpha}$ norm $\|u\|_{C^{2,\alpha}(\ov \Om)}$ depend on $\Omega,\beta$, $\min_{\ov\Om} f$, $\|f\|_{C^{\beta}(\ov \Om)}$, $\|\varphi\|_{C^{2,\beta}(\p \Om)}$, 
$\|\underbar{u}\|_{C^{2}(\ov \Om)}$ and the differences $\det D^2 \underbar{u} - f$ at the vertices of $\Omega$.
\end{thm}
\begin{rem}
\label{R1}
If we relax the assumption on $\underbar{u}$ in Theorem \ref{T1} to be a subsolution (not necessarily strict), then we obtain $u \in C^2(\overline \Om)$. This follows from Theorem \ref{T3}.
\end{rem}

Theorem \ref{T1} establishes continuity estimates of the second derivatives for the solutions to the Monge-Amp\`ere equation (\ref{DP}) near the vertices of a domain with corners. Depending on the data, solutions might develop conical singularities at the corners where the Hessian matrix becomes unbounded. A necessary condition for the $C^2$ estimates is the existence of a classical convex subsolution with the same boundary data. By the results above, this condition turns out to be sufficient as well. This is in contrast with the case of second order linear elliptic equations where the regularity of solutions depends on the smallness of the angles at the vertices.

We also note that Theorem \ref{T1} cannot hold in $n\ge 3$ dimensions. For example, we can take $\Omega$ to be the unit cube $[0,1]^3 \subset \R^3$, $f\equiv c <1$, and $\varphi= |x|^2/2$ on $\p \Omega$. Then $u$ cannot be $C^2$ at the origin since otherwise the boundary data imposes $D^2u(0)=I$ hence $\det D^2 u(0)=1 \ne f(0)$.

An interesting feature of the $C^{2,\alpha}$ estimates for \eqref{DP} is that they are not stable under small perturbations 
of the data $\varphi$ and $f$. 
The $C^{2,\alpha}$ norm of the solution $u$ depends crucially on the $C^2$ norm of the subsolution $\underline u$ and on the differences $\det D^2 \underline u - f$ at the vertices of $\Omega$. 
In fact we show that it is possible for $D^2u$ to oscillate of order 1 in an arbitrarily small neighborhood of a vertex when $\det D^2 \underline u$ and $f$ are allowed to be sufficiently close at that vertex. A more accurate analysis about the possible behaviors of solutions near a corner under general data is given at the end in Theorem \ref{T3}.

We prove Theorem \ref{T1} by first classifying global solutions to the Monge-Amp\`ere equation in the first quadrant in the plane with constant right hand side and quadratic boundary data. Our classification can be viewed as a Liouville type result for the Monge-Amp\`ere equation in angles in the plane.
Liouville type theorems for the Monge-Amp\`ere equation which state that global solutions must be quadratic polynomials are known in all dimensions if the domain is either the whole space or a half-space; see \cite{CL, S2}.

At a vertex of the polygon the solution $u$ to (\ref{DP}) is pointwise $C^{1,1}$ since it is bounded above by the convex function generated by the boundary data $\varphi$ and bounded below by the tangent plane of $\underline u$, which is also the tangent plane for the upper barrier. 
Using the affine invariance of the Monge-Amp\`ere equation (see \cite{F, Gu}), we may assume after an affine transformation that $\Omega$ is given by the first quadrant
$$Q:=\{x=(x_1, x_2)\in \R^2: x_1, x_2>0\},$$
in a neighborhood of the origin, and $\varphi_{x_1 x_1}(0)=\varphi_{x_2 x_2}(0)=1$.
Then a quadratic blow-up of the solution must converge to a global convex solution defined in the first quadrant $Q$
that satisfies
\begin{equation}
\label{ueq}
\det D^2 u=c, \quad \mbox{and} \quad u \ge 0, \quad \mbox{in $Q$},
\end{equation}
for some constant $c>0$, and
\begin{equation}
\label{ubdr}
u(x)=\frac{|x|^2}{2}~\text{on }\p Q.
\end{equation}

We denote by $P_c^\pm$ the quadratic polynomials that solve (\ref{ueq})-(\ref{ubdr}) when $0<c<1$ which are important in our analysis
$$P_c^{\pm}(x):= \frac{1}{2}x_1^2 + \frac{1}{2}x_2^2 \pm \sqrt{1-c}~ x_1x_2.$$

Our second main result classifies global convex solutions $u\geq 0$ of the Monge-Amp\`ere equation  in the first quadrant in the plane with quadratic boundary data and constant right hand side.

\begin{thm}\label{T2} Assume that $u$ is a solution to (\ref{ueq})-(\ref{ubdr}). Then $c \le 1$ and
\begin{myindentpar}{1cm}
(i) if $c=1$ then the only solution $u$ to (\ref{ueq})-(\ref{ubdr})  is $$u(x)=\frac{|x|^2}{2}.$$
(ii) if $c<1$ then either $$u=P_c^\pm, \quad \mbox{ or} \quad  u(x)= \lambda^{2}\, \, \bar P_c \left ( \frac x \lambda\right),$$ for some $\lambda \in (0, \infty)$ where $\bar P_c$ is a particular solution to (\ref{ueq})-(\ref{ubdr}) that satisfies $$P_c^- < \bar P_c < P_c^+ \quad \text{in } Q, \quad \text{and  } \bar P_c(1,1)=1.$$
Moreover, $\bar P_c\in C^{2,\alpha}(\overline{Q})$ for some $\alpha=\alpha(c)>0$, and 
$$\mbox{$\bar P_c(x)=P_c^+(x) + O(|x|^{2+\alpha})$ near $x=0$  and $\bar P_c(x)=P_c^-(x) + O(|x|^{2-\alpha})$ for all large $|x|$, }$$ hence $\bar P_c$ interpolates between the quadratic polynomial $P_c^+$ near $0$ and $P_c^-$ at $\infty$.
\end{myindentpar}
\end{thm}

Theorem \ref{T2} shows that any small positive perturbation of $P_c^-$ on $\p B_1\cap Q$, for example a rescaling of $\bar P_c$ for small $\lambda$, produces an arbitrarily large $C^{2,\alpha}$ norm near the origin. 

In Proposition \ref{Cl} we give more precise information when $c<1$ and classify all global solutions which do not necessarily satisfy the assumption $u \ge 0$. 
We show that there is a second family of solutions generated by quadratic rescalings of a particular solution $\underbar P_c$ of (\ref{ueq})-(\ref{ubdr})   which has a conical singularity 
at the origin. 

The rest of the paper is organized as follows. In Section \ref{C2_sect}, we state a compactness result and  derive second derivative 
estimates for global solutions. In Section 3 we establish pointwise $C^{2,\alpha}$ estimates for perturbations of the quadratic polynomials $P_c^\pm$.  The classification of global solutions is obtained in Section 4.
The final section, Section \ref{T1_pf}, will be devoted to proving the global $C^{2,\alpha}$ estimates in Theorem \ref{T1}.

 \section{Compactness and second derivative estimates for global solutions}
 \label{C2_sect}

 In this section, we obtain second derivative estimates and their consequences in the analysis of solutions to the Monge-Amp\`ere equation $\det D^2 u= c$  in the first quadrant in the plane with quadratic boundary data.
 
 \subsection{Compactness}
\label{Compactness} 
Assume that $u$ satisfies (\ref{ueq}) and (\ref{ubdr}).
 
 As mentioned in the Introduction, for $x=(x_1, x_2)\in Q$, we have from the convexity of $u$
that
$$u(x) \leq  \frac{x_1}{x_1+ x_2}u (x_1+ x_2, 0) + \frac{x_2}{x_1+ x_2}u(0, x_1+ x_2)=\frac{1}{2}(x_1+ x_2)^2\leq |x|^2.$$
Since $u\geq 0$, we can use standard barriers at points on $\p Q$ to obtain
$$|\nabla u| \le C(c) \quad \mbox{in} \quad (B_3 \setminus B_{1/3}) \cap  Q.$$ 
The function $u$ separates quadratically from its tangent plane on $\p Q$, so by the results in \cite[Theorem 6.4]{S1}, we find
$$\|u\|_{C^{3}} \le C_0(c) \quad \mbox{in} \quad (B_2\setminus B_{1/2}) \cap Q.$$
Applying the above estimate to the quadratic rescalings of $u$ (that is, those of the form $r^{-2} u(rx)$), we find
\begin{equation}
\label{uC2}
c_0(c)I \le  D^2 u \le C_0(c)I \quad \mbox{in} \quad Q,
\end{equation}
thus the Monge-Amp\`ere operator $\det D^2 u$ is uniformly elliptic, and
\begin{equation}
\label{uC3}
|D^3 u(x)| \le C_0(c) |x|^{-1}  \quad \mbox{in} \quad Q.
\end{equation}

The above estimates easily give the compactness in $C^3_{loc}(\ov Q \setminus \{0\})$ for a sequence of solutions to \eqref{ueq}-\eqref{ubdr} which we state below.

\begin{lemma}(\bf Compactness) {\it Let $u_k$ be a sequence of solutions to \eqref{ueq}-\eqref{ubdr}. Then, there exists a subsequence which converges (in the $C^3$ norm) on compact sets of $\ov Q \setminus \{0\}$ to another solution $u_\infty$ of \eqref{ueq}- \eqref{ubdr}.}
\label{comp}
\end{lemma}

\subsection{$C^{1,1}$ estimates }
Our first result is a sharp upper bound for the Hessian matrix $D^2u$.
\begin{lemma}
\label{uxx_lem}
Let $u$ be a convex function satisfying (\ref{ueq}) and (\ref{ubdr}). 
Then, for all $x\in Q$, we have
$$u_{x_1x_1}(x)\leq 1, u_{x_2 x_2}(x)\leq 1~\text{and } |u_{x_1x_2}(x)|\leq \sqrt{1-c}.$$
Thus, if $c>1$, then there are no solutions $u$ to (\ref{ueq}) and (\ref{ubdr}). If $c=1$ then the only solution to (\ref{ueq}) and (\ref{ubdr}) is $u(x)=\frac{|x|^2}{2}.$
\end{lemma}

We use the following notation for $1\leq i, j, k\leq 2$:
$$u_{ij}:= u_{x_i x_j}, \quad u_{ijk}:= u_{x_i x_j x_k}.$$

\begin{proof}It suffices to prove  $0\leq u_{11}\leq 1$. Then by symmetry $0\leq u_{22}\leq 1$, and $ |u_{12}|\leq \sqrt{1-c}$ follows from $u_{12}^2=u_{11} u_{22}-c$. \\
{\it Step 1:} We show that if $u_{11}$ attains its maximum value $M>1$ at some $p\in \overline{Q} \setminus\{0\}$ then we will get a contradiction. Indeed, suppose that $u_{11}$ attains its maximum value $M>1$ at $p.$
First, since $u_{11}$ is a 
subsolution of the linearized operator of $\det D^2 u$, 
$p$ must be on the boundary. Because $u_{11}=1$ on the $x_1$-axis and $u_{11}(p)=M>1$, we find that $p$ must be on the positive $x_2$-axis. It follows that 
\begin{equation}
\label{uxxy}
u_{112}(p)=0
\end{equation}
We claim that 
\begin{equation}\label{u12p0}
u_{122}(p)= 0.
\end{equation} 
Indeed, differentiating both sides of the equation (\ref{ueq}), that is $u_{11} u_{22}-u_{12}^2=c$,  with respect to $x_2$, we get
\begin{equation}
\label{dyc}
u_{112} u_{22} + u_{11} u_{222} -2 u_{12} u_{122}=0.
\end{equation}
Since $u_{112}(p)=u_{222}(p)=0$ we find that either $u_{122} (p)=0$ and we are done or $u_{12}(p)=0$. In the second case, on the $x_2$-axis, we have from (\ref{ueq}) and $u_{22}=1$ that $u_{12}^2= u_{11}-c$. The maximality of $u_{11}$ at $p$ shows that, on the $x_2$-axis, $u_{12}^2$ attains its maximum value at $p$. Thus, from 
$u^2_{12}(p)=0$, we find that $u_{12}=0$ on the whole $x_2$-axis, hence $u_{122}(p)=0$ and the claim is proved.

Differentiating both sides of the equation (\ref{ueq}) with respect to $x_1$, we find that
\begin{equation}
\label{dxc}
u_{111} u_{22} + u_{11} u_{122} -2 u_{12} u_{112}=0.
\end{equation}
Evaluating (\ref{dxc}) at $p$ using (\ref{uxxy})-\eqref{u12p0}, we find
 $u_{111} (p)=0$. This contradicts the Hopf maximum principle since $u_{11}$ is a nonconstant subsolution for the linearized equation.
 
 {\it Step 2:} We finally prove that if
 $M:= \sup_{Q} u_{11}$
 then $M\leq 1$. We argue by contradiction. Suppose that $M>1$. From the definition of $M$, there exists a sequence $\{ z_k\}\subset Q\setminus\{0\}$
 such that $u_{11}(z_k)\rightarrow M$ when $k\rightarrow\infty$. Let us define $$r_k=|z_k|, \quad z'_k= r_k^{-1} z_k\quad \text{and }v_k(z):= r_k^{-2} u(r_k z).$$ 
 Then, $v_k$ is a solution to (\ref{ueq})-(\ref{ubdr}); 
 moreover,
 $$v_{k, 11} (z'_k)= u_{11}(z_k)\rightarrow M\quad \text{when }k\to\infty.$$
 By Lemma \ref{comp}, the functions $v_k$ has a limit $v$ in $C^{3}_{loc}(\overline{Q})$ solving (\ref{ueq})-(\ref{ubdr})
  and at any limit point $z_\infty\in S^1\cap \overline{Q}$ of $z'_k$, the function $v_{11}$ attains its maximum value $M>1$. This contradicts Step 1.

 \end{proof}

From now on, in view of Lemma \ref{uxx_lem} we consider only the case $$0<c<1.$$ 
Before we proceed further we state a general result about mixed second partial derivative of solutions to fully nonlinear elliptic equations in two dimensions.
 \begin{lemma}
 \label{uxy_lem}
In two dimensions, if $u \in C^4$ solves the fully nonlinear elliptic equation $F(D^2 u)=0$, with $F\in C^2(\mathcal{S})$ where $\mathcal{S}$ is the space of real $2\times 2$ symmetric matrices, then
$u_{12}$ is a solution to a second order linear elliptic equation with no zero order terms. 
\end{lemma}

 \begin{proof}
Let us denote for each $r= (r_{ij})_{1\leq i, j\leq 2}\in \mathcal{S}$
$$F_{ij}:=\frac{\p F(r)}{\p r_{ij}}.$$
Differentiating both sides of $F(D^2 u)=0$ with respect to $x_1$, we get
\begin{equation}
\label{dxF}
F_{ij} u_{1ij}=0. 
\end{equation}
Differentiating both sides of the above equation with respect to $x_2$, we find that
$$F_{ij} (u_{12})_{ij}= - F_{ij, kl} u_{1ij} u_{2kl}.$$
The only term in the above right hand side that does not involve $u_{12}$ is $-F_{11,22} u_{111} u_{222}$. Note that, from (\ref{dxF}), we have
$F_{11} u_{111}= a_k u_{12k}$ for continuous functions $a_1$ and $a_2$,
and therefore
$$-F_{11,22} u_{111} u_{222} = -\frac{a_k u_{12k}}{ F_{11}} F_{11,22} u_{222}.$$
The result follows.

\end{proof}

Our final result of this section is concerned with possible limit values of the mixed second partial derivative of solutions to (\ref{ueq}) and (\ref{ubdr}).
\begin{lemma}
\label{uxy_bound}
Let $u$ be a convex function satisfying (\ref{ueq}) and (\ref{ubdr}). Then
\begin{myindentpar}{1cm}
(i) if $u_{12}$ achieves a local minimum or maximum at some point in $\overline{Q}\setminus\{0\}$ then $u=P_c^\pm$. \\
(ii) we have
$$ \liminf_Q u_{12}, \, \, \limsup_Q u_{12} \quad \in \, \{\pm \sqrt{1-c} \}.$$ 
\end{myindentpar}
  \end{lemma}
 
 In particular if $u_{12}= \pm \sqrt{1-c}$ at some point in $\overline{Q}\setminus\{0\}$ then we have $u=P_c^{\pm}$. By compactness we obtain:
 \begin{corollary}
 \label{cor_pm}
 Let $u$ be a convex function satisfying (\ref{ueq}) and (\ref{ubdr}). 
 \begin{myindentpar}{1cm}
 (i) If $u_{12} \le -\sqrt {1-c} + \delta$ at some point on $\p B_1 \cap \bar Q$ then 
 \begin{equation}\label{200}
 \|u-P_c^-\|_{C^2} \le \eps \quad \mbox{in} \quad ( B_{1/\rho}\setminus B_\rho) \cap \bar Q
 \end{equation}
 for some  $\eps(\delta)>0$ and $\rho(\delta)>0$ small,
 and $\eps(\delta) \to 0$, $\rho(\delta) \to 0$ as $\delta \to 0$.\\
 (ii)  Similarly, if $u_{12} \ge \sqrt {1-c} - \delta$ at some point on $\p B_1 \cap \bar Q$ then 
  \begin{equation}\label{201}
  \|u-P_c^+\|_{L^\infty} \le \eps \quad \mbox{in} \quad B_2\cap \bar Q,
  \end{equation}
   for some  $\eps(\delta)>0$ small,
 and $\eps(\delta) \to 0$ as $\delta \to 0$.
\end{myindentpar}
\end{corollary}

\begin{rem}\label{dfl}
As a consequence of the above results we find that either $u=P_c^\pm$ or $u_{12}$ has different limits $\pm \sqrt{1-c}$ at $0$ and $\infty$. 
\end{rem}
We will show, using the $C^{2,\alpha}$ estimates in the next section that, for any nonquadratic solution $u$ to (\ref{ueq})-(\ref{ubdr}), $\sqrt{1-c}$ must be the limit at $0$ and $-\sqrt{1-c}$ the limit at $\infty$
for $u_{12}$; see Lemma \ref{Pcpm2}.

\begin{proof}[Proof of Lemma \ref{uxy_bound}]
We prove (i) by showing that if $u_{12}$ has a local minimum or a local maximum  in $\overline{Q}\setminus \{0\}$ then it is a constant which is $\pm\sqrt{1-c}$.
Suppose that $u_{12}$ is not a constant in $\overline{Q}$. Then, by Lemma \ref{uxy_lem} applied to the equation $F(D^2 u):=\det D^2 u-c=0$, we deduce that the extreme point of $u_{12}$
must be on the boundary, say at $(0,1)$ on the $x_2$-axis. At this point, we use (\ref{dyc}) to obtain that $u_{112}=0$. But this is exactly $(u_{12})_{x_1}=0$ so, by Lemma \ref{uxy_lem}, we contradict the Hopf lemma.

Since $u_{12}$ is a constant $\lambda$, then $u=\lambda xy + f(x)+g(y)$ and then we find $u=P_c^\pm$.

Now, we prove the two assertions in (ii) which follows easily from (i) and compactness using quadratic rescalings. Let 
\begin{equation}
\label{a_inf}
a:=  \liminf_Q u_{12}.
\end{equation}
Then, by Lemma \ref{uxx_lem}, we have $a\geq -\sqrt{1-c}$. Moreover, there is a sequence $\{z_k\}_{k=1}^{\infty} \subset Q\setminus\{0\}$ such that
$u_{12}(z_k)\rightarrow a$ when $k\rightarrow \infty$. Let $r_k=|z_k|$ and $z_k^{'}= r_k^{-1}z_k$. Define
$$v_k(z)= r_k^{-2} u(r_k z).$$
Then $v_k$ satisfies  (\ref{ueq})-(\ref{ubdr}), $ v_{k,12} \ge a$ and
$$v_{k, 12}(z_k^{'})= u_{12}(z_k)\rightarrow a\quad\text{as }k\to\infty.$$

By the compactness result of Lemma \ref{comp}, there exists a subsequence of $\{v_k\}$, still denoted $\{v_k\}$, which converges (in the $C^3$ norm) on compact sets of $\ov Q \setminus \{0\}$ to another solution $v$ of \eqref{ueq}- \eqref{ubdr}. Moreover, we can also assume (after relabeling a subsequence) that
 $z_k^{'}\rightarrow z \in \p B_1 \cap \overline {Q}$. We have
 $$v_{12}(z)=a,$$
and $v_{12} \ge a$ in $Q$. The fact that $a\in\{\pm \sqrt{1-c}\}$ follows from (i).

\end{proof}

\section{Pointwise $C^{2,\alpha}$ estimates}

In this section we prove pointwise $C^{2,\alpha}$ estimates at the origin for solutions of the Monge-Amp\`ere equation in the first quadrant in the plane which are perturbations of $P_c^\pm$.  
 
 Following \cite{CC}, we say that $u$ is $C^{2, \alpha}$ at $x_0$, and write $u\in C^{2,\alpha}(x_0)$, if there exists a quadratic polynomial $P_{x_0}$ such that, in the domain of definition of $u$, $$u(x)=P_{x_0}(x) + O(|x-x_0|^{2+\alpha}).$$

Assume that the convex function $u$ solves the following Dirichlet problem for the Monge-Amp\`ere equation 
\begin{equation}\label{DP2}
\det D^2 u=f \quad \mbox{in $Q$}, \quad u=\varphi \quad \mbox{on $ \p Q$}.
\end{equation} 
We prove the following pointwise $C^{2,\alpha}$ estimates when $f$ is close to $c$ and $\varphi$ to $|x|^2/2$.
For simplicity of notation we use $q$ for this quadratic data, that is,
$$q(x):=\frac{|x|^2}{2}.$$

\begin{propo}
\label{C2P+}
Let $c\in (0, 1)$.
 Assume that $u$ satisfies \eqref{DP2} and suppose that 
 $$|u-P_c^{+}|\leq \eps \quad\text{and} \quad  |f-c| \le \delta \eps \quad \mbox{in 
$B_1\cap Q$},
\quad \text{and } |\varphi - q| \le \delta \eps \quad \mbox{on
$B_1\cap \p Q$,}$$
 for some $\eps \le \eps_0(c)$ small and $\delta(c)$ small. Then there exist  $\alpha\in (0, 1)$ and $r\leq \frac{1}{2}$  depending only on $c$ such that 
$$|u-P_c^{+}|\leq \eps r^{2+\alpha}~\text{in } B_r\cap Q.$$
\end{propo}

If $f$ and $\varphi$ are pointwise $C^\alpha$ and $C^{2,\alpha}$ respectively, then we can apply Proposition \ref{DP2} indefinitely and obtain the pointwise $C^{2,\alpha}$ estimate for $u$ at the origin.
\begin{corollary}\label{Pc2a}
Let $c\in (0, 1)$.
 Assume that $u$ satisfies \eqref{DP2} and suppose that 
 $$|u-P_c^{+}|\leq \eps_0 \text{ and } |f(x)-c| \le \delta \eps_0 |x|^\alpha \quad \mbox{in 
$B_1\cap Q$},
\quad \text{ and } |\varphi(x) - q(x)| \le \delta \eps_0 |x|^{2+\alpha} \quad \mbox{on
$B_1\cap \p Q$,}$$
for some $\eps_0(c)$ small and $\delta(c)$ small.
Then
$$|u(x)-P_c^{+}(x)|\leq C \eps_0 |x|^{2+\alpha}~\text{in } B_1\cap Q.$$
\end{corollary}

This result shows that the only possible limit for $u_{x_1 x_2}(x)$ as $x\to 0$ is $\sqrt{1-c}$ for any nonquadratic solution $u$ to (\ref{ueq})-(\ref{ubdr}). Indeed, by Lemma \ref{uxy_bound} and Corollary \ref{cor_pm}, \eqref{201} holds after an initial dilation for some $\eps \le \eps_0$, and then Proposition \ref{DP2} above applies indefinitely.

Our next proposition deals with the case when $u$ is close to $P_c^-$. We introduce the following exponent 
 $$\beta_c^-: =\frac{\pi}{\arccos (-\sqrt{1-c})} \in (1,2).$$

\begin{propo}
\label{C2P-}
Let $c\in (0, 1)$ and $\beta\in (\beta_c^-, 2]$.
Assume that $u$ satisfies \eqref{DP2} and suppose that 
 $$|u-P_c^{-}|\leq \eps |x|^\beta, \quad \text{and}\quad  |f-c| \le \delta \eps \quad \mbox{in 
$(B_{1/\rho} \setminus B_{\rho})\cap Q$},$$
and 
$$ |\varphi - q| \le \delta \eps \quad \mbox{on 
$(B_{1/\rho} \setminus B_{\rho})\cap \p Q$}$$
with $\eps \le \eps_0(c, \beta)$, $\delta=\delta(c,\beta)$, $\rho=\rho(c,\beta)$ small. Then
 $$|u-P_c^{-}|\leq \frac {\eps}{2} \quad \mbox{on $\p B_1\cap Q$.}$$ 
\end{propo}

A consequence of this result is that if $u$ is quadratically close to $P_c^-$ at all scales less than $1$, i.e.,
$$|u(x)-P_c^-(x)| \le \eps_0|x|^2, |f(x)-c| \le \delta \eps_0 |x|^\alpha  \mbox{ in $Q \cap B_1$}, \text{ and } |\varphi(x) -q(x)| \le \delta \eps_0 |x|^{2+\alpha} \mbox{ on $\p Q \cap B_1$},$$  for some 
$\eps_0,\delta, \alpha\in (0,1)$ small depending only on $c$,
then $$|u(x)-P_c^-(x)| \le C \eps_0 |x|^{2+\alpha}$$ near the origin; see Lemma \ref{Pc-cor}.

\subsection{Transformed domains $Q_c^{\pm}$ and reformulations of Propositions \ref{C2P+} and \ref{C2P-}}
\label{Qcpm}
We use affine transformations to transform $P_c^{\pm}$ into the quadratic function 
$q(x)=\frac{|x|^2}{2}$ on appropriate angular domains $Q_c^{\pm}$ in the plane. Then the linearized operator of $\det D^2 u$ around $q$ is the Laplace operator. 
We assume that $u$  satisfies \eqref{DP2} and the hypotheses of either Proposition \ref{C2P+}  or Proposition \ref{C2P-}.
We start with the affine transformations from $\R^2$ to $\R^2$ given by the matrices
\begin{equation*}
A_c^{\pm}:=
\begin{pmatrix} 1 & \mp \frac{\sqrt{1-c}}{\sqrt{c}}\\ 0 &\frac{1}{\sqrt{c}} \end{pmatrix}
\quad\text{and } (A_c^{\pm})^{-1}=
\begin{pmatrix} 1 & \pm \sqrt{1-c}\\ 0 & \sqrt{c} \end{pmatrix},
\end{equation*}
and denote 
$$Q_c^{\pm}= (A_c^{\pm})^{-1}Q,~u_c^{\pm}= u \circ A_c^{\pm}, ~\text{and } q_c^{\pm}= q \circ A_c^{\pm}.$$
Then $$P_c^{\pm} \circ A_c^{\pm}(x)= q(x) =\frac{|x|^2}{2} \quad \text{on } Q_c^{\pm}.$$
Note that
$$\det D^2 u_c^{\pm}=\frac{1}{c} f\circ A_c^{\pm}~\text{and } |\det D^2 u_c^{\pm} -1| = \frac{|f-c|}{c}\leq \frac{\ep\delta}{c},$$
and
$$q_c^{\pm}(x)=P_c^\pm\circ A_c^{\pm}(x)=q(x)=\frac{|x|^2}{2}\quad \text{on }\p Q_c^{\pm}.$$

We restate equivalent versions of Proposition \ref{C2P+} and Proposition \ref{C2P-} on the transformed domains $Q_c^{\pm}$ as follows.
\begin{propo}
\label{C2P+2}
Suppose that $|\det D^2 u-1|\leq \delta \ep$, $|u-q|\leq\ep\leq \ep_0$ in $B_1\cap Q_c^{+}$ where $0<\ep_0(c),\delta(c)\leq \frac{1}{16}$ are sufficiently small and 
 $u$ has the boundary value $\varphi$  on the edges of $Q_c^{\pm}$ that satisfies
$|\varphi -q|\leq \delta \ep$ on $B_1\cap\p Q_c^+$.
 Then there exist $\alpha\in (0, 1)$ and  $r\in (0,\frac{1}{2})$ depending only on $c$ such that
$$|u-q|\leq \ep r^{2+\alpha}~\text{in } B_r\cap Q_c^{+}.$$
\end{propo}
\begin{propo}
\label{C2P-2}
Let $\beta\in (\beta_c^-,2]$.
Suppose that $|\det D^2 u-1|\leq  \delta \ep $, $|u(x)-q(x)|\leq\ep |x|^\beta$ in $Q_c^-\cap (B_{1/\rho}\setminus B_\rho)$ where $0<\ep\leq \ep_0(c,\beta),\delta(c,\beta), \rho(c,\beta)\leq \frac{1}{16}$ are sufficiently small and 
 $u$ has the boundary value $\varphi$  on the edges of $Q_c^{-}$ that satisfies
 $|\varphi -q|\leq \delta \ep$ on $ (B_{1/\rho}\setminus B_\rho)\cap\p Q_c^-$.
 Then 
$$|u-q|\leq \frac{\ep}{2}~\text{on } \p B_1\cap Q_c^{-}.$$
\end{propo}

To prove these propositions, we show that the ratio $\frac{u-q}{\ep}$ is well approximated by a harmonic function on $Q_c^{\pm}$ which vanishes on the boundary.
The approximation results state as follows.
\begin{lemma}
\label{appro_lem}
Assume that $u$ satisfies the hypotheses of Proposition \ref{C2P+2}.
Then, for any small $\eta>0$, we can find a solution $w$ to
\begin{equation}
\label{w_har+}
\Delta w =0 \text{ in } Q_c^{+},~w=0~\text{ on } \p Q_c^{+} 
\end{equation}
such that $|w|\leq 1$ in $B_{1/2}\cap\overline{Q_c^{+}}$ and
$$|u-q-\ep w| \leq \ep\eta~\text{ in } B_{1/2}\cap\overline{Q_c^{+}}$$
provided that $\ep_0(\eta, c)$ and $\delta (\eta, c)$ are chosen sufficiently small, now depending also on $\eta$.
\end{lemma}
\begin{lemma}
\label{appro_lem-}
Assume that $u$ satisfies the hypotheses of Proposition \ref{C2P-2}.
Then, for any small $\eta>0$, we can find a solution $w$ to
\begin{equation}
\label{w_har-}
\Delta w =0 \text{ in } Q_c^{-},~w=0~\text{ on } \p Q_c^{-} 
\end{equation}
such that $|w|\leq |x|^\beta$ in $(B_{1/(2\rho)}\setminus B_{2\rho})\cap\overline{Q_c^{-}}$ and
$$|u-q-\ep w| \leq \ep\eta~\text{ in } (B_{1/(2\rho)}\setminus B_{2\rho})\cap\overline{Q_c^{-}}$$
provided that $\ep_0(\eta, c)$ and $\delta(\eta, c)$ are chosen sufficiently small, now depending also on $\eta$.
\end{lemma}

\begin{proof}[Proof of lemma \ref{appro_lem}] The proof of this lemma is similar to that of Lemma 2.6 in \cite{LS2017}. We give the details below. First we show that in $B_{1/2} \cap Q_c^+$ we have
\begin{equation}\label{300}
|u(x)-q(x)| \le C \ep\dist(x,\p Q_c^+) +  \delta \ep,
\end{equation}
for some constant $C$ depending only on $c$. Pick a point $(a,0)$ on the $x_1$ - axis, with $ a \in [0, 1/2]$. We claim that
$$\bar w := q + \delta \eps + 4 \ep [(x_1-a)^2 - 2 x_2^2] + C \ep x_2,$$
is an upper barrier for $u$ in $B_1 \cap Q_c^+$, and
$$ \underline w := q - \delta \eps - 4 \ep [(x_1-a)^2 - 2 x_2^2] - C \ep x_2,$$
is a lower barrier. Indeed, $$\det D^2 \bar w \le  1 - \eps \leq \det D^2 u,$$
and 
$$\mbox{$\bar w \ge q + \delta \eps\ge u$ on $\p Q_c^+ \cap B_1$,  and $\bar w \ge q + \eps\ge u$ on $\p B_1\cap Q_c^+$,}$$
provided that $C$ is chosen sufficiently large. Thus $u \le \bar w$ in  $B_1 \cap Q_c^+$ by the maximum principle. Similarly we obtain that $u \ge \underline w$ in $B_1 \cap Q_c^+$. By choosing $a=x_1$, we find 
$$|u(x)-q(x)| \le C' \eps x_2 + \delta \eps \quad \mbox{ in} \quad  B_1 \cap Q_c^+ \cap \{0<x_1< 1/2\},$$ 
and \eqref{300} easily follows.

Next we define
$$v_{\ep}:= (u-q)/\ep,$$ 
and, by hypothesis, 
$$|v_\ep|\leq 1 \quad \text{in}\quad B_1\cap Q_c^{+}.$$
It suffices to show that for a sequence of $\ep, \delta \rightarrow 0$, the corresponding $v_{\ep}$'s converges uniformly in $B_{1/2} \cap\overline{Q_c^{+}}$ to a solution of (\ref{w_har+}) along a subsequence.

By \eqref{300} we find that $v_{\ep}$ grows at most linearly away from $\p Q_c^{+}$. 

It remains to prove the uniform convergence of $v_{\ep}$'s on compact subsets of $B_{1}\cap Q_c^{+}$. 

Fix a ball $B_{2r}(z)\subset B_{1}\cap Q_c^{+}$.  Let $u_0$ be the convex solution to $\det D^2 u_0 =1$ in $B_{2r}(z)$ with boundary value $u_0=u$ on $\p B_{2r}(z)$. We claim that
\begin{equation*}|u-u_0|\leq 4r^2 \delta \ep~\text{in } B_{2r}(z).
\end{equation*}
To see this, we use the maximum principle and the following inequality
$$
\det (A+\lambda I_2) \geq \det A + 2\lambda (\det A)^{1/2},~\text{if } A\geq 0, \lambda\geq 0$$
to obtain in $B_{2r}(z)$ 
$$u+ \delta \ep (|x-z|^2 -(2r)^2) \leq u_0\quad \text{and } u_0 + \delta \ep (|x-z|^2 -(2r)^2)\leq u$$
from which the claim follows.

Now, if we denote
$$v_0:= (u_0-q)/\ep$$
then
$$|v_\ep-v_0|=|u-u_0|/\ep\leq 4r^2\delta~\text{in } B_{2r}(z),$$
and hence $v_\ep-v_0\rightarrow 0$ uniformly in $\overline{B_r(z)}$ as $\delta \rightarrow 0$.

Next, we show that, as $\ep_0\rightarrow 0$, the corresponding $v_0's$ converges uniformly, up to extracting a subsequence, in $\overline{B_r(z)}$, to a solution of (\ref{w_har+}).
Note that
$$0= \frac{1}{\ep}(\det D^2 u_0-\det D^2 q)= \text{trace} (A_\ep D^2 v_0)$$
where,  using $\text{cof}(M)$ to denote the cofactor matrix of the matrix $M$,
$$A_\ep= \int_0^1 \text{cof} (D^2 q + t (D^2 u_0-D^2 q)) dt.$$

We note that as $\ep_0\rightarrow 0$, we have $\ep\rightarrow 0$ and $u\to q$; therefore $D^2 u_0\rightarrow D^2 q= I_2$ uniformly in $\overline{B_r(z)}$. This shows 
that $A_\ep\rightarrow I_2$ uniformly in $\overline{B_r(z)}$ and thus $v_0$'s must converge to a harmonic function $w$ satisfying (\ref{w_har+}).
The bound $|w|\leq 1$ in $B_1\cap Q_c^{+}$ follows from from the corresponding bound for $v_\ep$ and the convergence $v_\ep-v_0\rightarrow 0$.
\end{proof}
\begin{proof}[Proof of Lemma \ref{appro_lem-}]
The proof of this lemma is essentially the same as that of Lemma \ref{appro_lem} so we omit it.

\end{proof}

\subsection{Harmonic functions in $Q_c^{\pm}$}
\label{Harmonic_sect}
Next we collect some standard facts about harmonic functions which vanish on the boundary of an angle. We note that, at the vertex $0$, the opening of $Q_c^{+}$ is an acute angle $\alpha_c^{+}\in (0,\frac{\pi}{2})$ while the opening of $Q_c^{-}$ is an obtuse angle $\alpha_c^{-}\in (\frac{\pi}{2},\pi)$ . In fact, we have
$$\cos \alpha_c^\pm = \pm \sqrt{1-c}.$$
Let us denote
$$\beta_c^{\pm}=\frac{\pi}{\alpha_c^{\pm}}.$$
Note that
$$\beta_c^+ >2~\text{while } 1<\beta_c^-<2.$$
 For any $(x_1, x_2)\in \R^2$, we can identify it with the complex number $z= x_1+ i x_2\in \C$. The conformal mappings $z\in Q_c^{\pm}\rightarrow \hat z^\pm:= z^{\beta_c^{\pm}}\in \H$ map
$Q_c^\pm$ to the upper-half plane $\H$. 
Let us consider $\hat w^\pm (\hat z^\pm)= w(z)$.
Corresponding to any solution $w$ to
$$\Delta w =0 \text{ in }Q_c^{\pm },~w=0~\text{ on } \p Q_c^{\pm}, $$
there is a harmonic function $\hat w$ in the upper-half plane $\H$ with zero boundary data, that is, $\hat w=0$ on $\p \H=\{x_2=0\}.$ Moreover, $w$ can be recovered from $\hat w$ via the formula
$$w(z) = \hat w (z^{\beta_c^\pm}).$$
As such, 
any solution $w$ to
$$\Delta w =0 \text{ in }Q_c^{+},~w=0~\text{ on } \p Q_c^{+} $$
is $C^{2,\alpha}$ in $\overline{B}_{1/2}\cap \overline{Q}_c^{+}$ for any $\alpha\in (0, 1]$ satisfying $\alpha\leq \beta_c^+-2$.

\begin{lemma}
\label{C2_Taylor}
Assume that $w$ solves 
\begin{equation}
\label{w_har2}
\Delta w =0 \text{ in } Q_c^{+},~w=0~\text{ on } \p Q_c^{+} 
\end{equation}
and $\|w\|_{L^{\infty}(B_1\cap Q_c^{+})}\leq 1$. Then there are constants $C_0>0$ and $\alpha_0\in (0, 1)$ depending only on $c$ such that $w$ satisfies
$$|w(x)|\leq C_0|x|^{2+\alpha_0}\quad \text{in }B_{1/2}\cap Q_c^{+}.$$
\end{lemma}
\begin{proof}
Note that the harmonic function $\hat w$ corresponding to $w$ is smooth in $B_{3/4}\cap \overline{\H}$. Thus, we have $$\|D \hat w \|_{L^{\infty}(B_{3/4}\cap\overline{\H})}\leq C.$$ It follows that for any $\hat z\in B_{3/4}\cap\overline{\H}$, we have
$$|\hat w(\hat z)|=|\hat w(\hat z)-\hat w(0)|\leq C|\hat z|. $$
The desired estimate of the lemma with $\alpha_0:=\min\{1, \beta_c^+-2\}$ follows from
$w(z) = \hat w (z^{\beta_c^+}).$

\end{proof}

A solution $v$ to
$$\Delta v =0 \text{ in } Q_c^{-},~v=0~\text{ on } \p Q_c^{-} $$
can be only $C^{1,\alpha}$ in $\overline{B}_{1/2}\cap \overline{Q}_c^{-}$. 

{\bf Notation.} We denote by $v_0=\text{Im}(z^{\beta_c^-})$ the positive, homogenous of degree $\beta^-_c \in (1,2)$ harmonic function which satisfies the equation above. In polar coordinates $(r,\theta)$, $v_0$ is given by
\begin{equation}
\label{v0fn}
v_0(r,\theta)=r^{\beta_c^-} \sin (\beta_c^-\theta).
\end{equation}

We need the following result for the proof of Proposition \ref{C2P-2}.
\begin{lemma}
\label{appro_lem2}
Let $\beta\in (\beta_c^-, 2 \beta_c^-)$.
Suppose that $w$ satisfies
\begin{equation}
\label{w_har3}
\Delta w =0 \text{ in }Q_c^{-},~w=0~\text{ on } \p Q_c^{-}, 
\end{equation}
and that $$|w(x)|\leq |x|^{\beta} \quad \mbox{ in}  \quad (B_{1/(2\rho)}\setminus B_{2\rho})\cap Q_c^{-}.$$ Then, given a positive constant $\gamma$, we can find $\rho=\rho(\beta,\gamma, c)>0$ 
sufficiently small such that $$|w|\leq \gamma \quad \mbox{ on}  \quad \p B_1 \cap Q_c^{-}.$$
\end{lemma}
\begin{proof} 
Let $\alpha:= \frac{\beta}{\beta_c^-}\in (1,2)$.
Using a conformal mapping to transform $Q_c^-$ to the upper half-plane $\H$, the statement of the lemma is equivalent to the following statement:\\
{\it Let $\alpha\in (1,2)$.
Suppose that $ w$ satisfies
\begin{equation}
\label{w_har4}
\Delta  w =0 \text{ in } \H,~w=0~\text{ on } \{x_2=0\}
\end{equation}
and that $|w(x)|\leq |x|^{\alpha}$ in $(B_{1/(2\rho)}\setminus B_{2\rho})\cap \H$. Then, given a positive constant $\gamma$, we can find $\rho=\rho(\alpha,\gamma)>0$ sufficiently small such that $|w|\leq \gamma$ on $\p B_1 \cap \H$.}

Suppose that the conclusion is false for some $\alpha_0\in (1,2)$. Thus, for each positive integer $n$, we can find a harmonic function $v_n$ in $(B_{n}\setminus B_{1/n})\cap \H$
with $v=0$ on $(B_{n}\setminus B_{1/n})\cap \p \H$ and $|v_n(x)|\leq |x|^{\alpha_0}$ in $(B_{n}\setminus B_{1/n})\cap \H$ but $\|v_n\|_{L^{\infty}(\p B_1\cap \H)}\geq \gamma$. 

Using compactness, we can let $n\rightarrow\infty$ along a subsequence 
to obtain a harmonic function $v$ on $\H$ with the following property:
$$v=0 \text{ on }\p \H, \quad |v(x)|\leq |x|^{\alpha_0} \text{ on }\H \quad \text{ and }\|v\|_{L^{\infty}(\p B_1\cap \H)}\geq \gamma.$$ By using refection about the $x_1$-axis and the Liouville theorem for harmonic functions with polynomial growth, we conclude that $v$ is at polynomial of degree almost 1. Thus, $v$ is of the form
$\pm C x_2$ for some positive constant $C$. Using
$$(x_1^2 + x_2^2)^{\alpha_0/2}\geq |v(x_1, x_2)|= C|x_2|$$
near the origin, we conclude that $C=0$. This contradicts  $\|v\|_{L^{\infty}(\p B_1\cap \H)}\geq \gamma$.

\end{proof}

\begin{rem}\label{r33}
The lemma above is true if we replace $|x|^\beta$ by $\max\{|x|^{\beta_1}, |x|^{\beta_2}\}$ where $\beta_1,\beta_2\in (\beta_c^-, 2\beta_c^-)$ satisfying $\beta_1\leq \beta\leq \beta_2$. 
This means that in Proposition \ref{C2P-} we can relax the hypothesis on $u-P_c^-$ to
$$|u-P_c^-| \le \eps \max\{|x|^{\beta_1}, |x|^{\beta_2}\} \quad \mbox{in} \quad Q \cap (B_{1/\rho} \setminus B_\rho)$$
 where $\beta_1,\beta_2\in (\beta_c^-, 2\beta_c^-)$ satisfying $\beta_1\leq \beta\leq \beta_2$.
 
It follows that if $\beta$ is bounded away from $\beta_c^-$ then we can choose $\rho(c,\beta)$ in Proposition \ref{C2P-} to be also bounded away from $0$.
\end{rem}

\subsection{Proofs of Propositions \ref{C2P+} and \ref{C2P-}}
They are reduced to those of Propositions \ref{C2P+2} and \ref{C2P-2} which we present in this section.
\begin{proof}[Proof of Proposition \ref{C2P+2}] Fix $\alpha\in (0,\alpha_0)$ where $\alpha_0$ is as in Lemma \ref{C2_Taylor}.
The proof, using Lemma \ref{appro_lem} and the $C^{2,\alpha_0}$ estimates for harmonic functions 
on $Q_c^{+}$ in Lemma \ref{C2_Taylor},
  is similar to the $C^{2,\alpha}$ estimates in \cite[Section 2]{LS2017}. We briefly indicate some details. 
  For any $\eta>0$, using Lemma \ref{appro_lem} and \ref{C2_Taylor}, we find that in $B_{1/2}\cap Q_c^{+}$
  $$|u(x)-q(x)|\leq \ep (\eta + C_0|x|^{2+\alpha_0})$$
  provided that $\ep_0(\eta, c)$ and $\delta(\eta, c)$ are chosen sufficiently small and $0<\ep\leq \ep_0(\eta, c)$.
  We choose $\eta = C_0 r_0^{2+\alpha_0}$ for some $r_0>0$ small to be chosen later. Then, in $B_{r_0}\cap Q^{+}_c$,
  $$|u-q|\leq 2\ep C_0 r_0^{2+\alpha_0}\leq \ep r_0^{2+\alpha}$$
  if $r_0$ is sufficiently small depending only on $c$ and $\alpha$. 
 \end{proof}

\begin{proof}[Proof of Proposition \ref{C2P-2}] 
Fix $\eta=\frac{1}{4}$. Let $w$ be as in the statement of Lemma \ref{appro_lem-}.
Then
$$|u-q|\leq \ep (\eta + |w|)~\text{ in } (B_{1/(2\rho)}\setminus B_{2\rho})\cap\overline{Q_c^{-}}$$
provided that $\ep_0(\beta, c)$ and $\delta(\beta, c)$ are chosen sufficiently small and $0<\ep\leq \ep_0(\beta, c)$.
Applying Lemma \ref{appro_lem2} to $w$ and $\gamma:= \frac{1}{2}-\eta=\frac{1}{4}$, we find that, 
$$|w|\leq \gamma\quad \text{on } \p B_1 \cap Q_c^-$$
provided that  $\rho=\rho(\beta, c)$ 
sufficiently small. Therefore, if $\ep_0(\beta, c),\delta(\beta, c)$ and $\rho(\beta, c)$ are sufficiently small, we have
$$|u-q|\leq \ep (\eta + |w|) \leq \ep (\eta + \gamma) =\frac{\ep}{2}\quad \text{on } \p B_1 \cap Q_c^-.$$

 \end{proof}
  
\subsection{Consequences of the second derivative estimates}  
\label{Cor_sect}
Next we state several consequences of the second derivative estimates in Corollary \ref{cor_pm} and Propositions \ref{C2P+} and \ref{C2P-}.

\begin{lemma} 
\label{Pcpm2}
Assume that $u$ is a solution to (\ref{ueq})-(\ref{ubdr}) which is not quadratic. Then
$$\lim_{x \to 0} u_{12}(x)=\sqrt{1-c}, \quad \mbox {and} \quad \lim_{|x| \to \infty} u_{12}(x)=-\sqrt{1-c}.$$
\end{lemma}

\begin{proof}
From Corollary \ref{cor_pm} and Corollary \ref{Pc2a} we know that if  $u_{12}(z)\geq \sqrt{1-c}-\delta$ at some point $z$ in  $\p B_r \cap Q$, with $\delta$ small 
universal, then $$|u(x)-P_c^+(x)| \le \eps_0 r^{-\alpha} |x|^{2+\alpha}  \quad \mbox{in} \quad B_{r/2} \cap Q.$$ This implies 
that $u_{12}$ converges to $\sqrt{1-c}$ at the origin and the lemma follows by Remark \ref{dfl}.

\end{proof}

\begin{lemma}
\label{Pc-cor}
Assume that $u$ satisfies (\ref{DP2}) where $c\in (0,1)$. Furthermore, assume that
  $$|u(x)-P_c^-(x)| \le \eps_0|x|^2, |f(x)-c| \le \delta \eps_0 |x|^\alpha \quad \mbox{in $Q \cap B_1$},\quad \text{and }
  |\varphi(x) -q(x)| \le \delta \eps_0 |x|^{2+\alpha}  \mbox{ on $\p Q \cap B_1$},$$
  where $\ep_0$, $\delta$, $\alpha$ are small depending on $c$. 
Then $$|u(x)-P_c^-(x)| \le C \eps_0 |x|^{2+\alpha}\quad \text{in } Q\cap B_1.$$ 
\end{lemma}
\begin{proof}
Let $\eps_0=\eps_0(c,2)$, $\delta(c,2)$, and $\rho= \rho(c, 2)$ be as in the statement of Proposition \ref{C2P-}. Choose $\alpha\in (0, 1)$ so that $\rho^\alpha =1/2$. Let $\delta= \delta(c, 2) \rho^{1+\alpha}$.
First, we claim that 
\begin{equation}
\label{ep02}
|u(z)- P_c^-(z)|\leq \frac{\ep_0}{2}|z|^2
\end{equation} for all $z\in Q$ satisfying $$|z|\leq \rho.$$
Indeed, let us fix $|z_0|=r\leq \rho$. We write $z_0= r x_0$ where $|x_0|=1$.  Consider the following functions 
$$\hat u (x) = r^{-2} u(rx),~\hat f (x) = f(rx),~\hat\varphi (x)= r^{-2}\varphi (rx).$$
Then on $B_{1/\rho}\cap Q$
$$|\hat u(x) -P_c^-(x)|= r^{-2}|u(rx)-P_c^-(rx)|\leq r^{-2} \ep_0 |rx|^2 =\ep_0 |x|^2,$$
and 
$$|\hat f(x) - c|= |f(rx)- c|\leq \delta \eps_0 r^\alpha |x|^\alpha,$$
and
$$|\hat \varphi(x) - q(x)|= r^{-2} |\varphi(rx) - q(rx)|\leq \delta \eps_0 r^\alpha |x|^{2+\alpha}.$$

Then $\hat u, \hat f$, and $\hat\varphi$ satisfy the hypotheses of Proposition \ref{C2P-} since $r\leq \rho\leq 1$. By this proposition, we have $|\hat u - P_c^-|\leq \frac{\ep_0}{2}$ on $\p B_1\cap Q$, hence
$$|u(z_0)-P_c^-(z_0)|\leq \frac{\ep_0}{2} r^2 = \frac{\ep_0}{2} |z_0|^2.$$
It follows by induction that
\begin{equation}
\label{ep02k}
|u(x)- P_c^-(x)|\leq \frac{\ep_0}{2^k}|x|^2~\text{for all } x\in Q ~\text{with } |x|\leq \rho^{k}.
\end{equation}
Indeed, as in (\ref{ep02}) we find 
$$|u(x)-P_c^-(x)|\leq \ep_k |x|^2~\text{for all } x\in Q~\text{with }|x|\leq r_k:= \rho^k,$$
with $\ep_k :=2^{-k}\ep_0$, and for this we used $ \eps_0 r_k ^\alpha =\eps_k $. 
The conclusion of the lemma now easily follows.
\end{proof}

\section{Classification of global solutions}
\label{T2_pf}
In this section, we prove Theorem \ref{T2} concerning classification of global solutions which satisfy  
\begin{equation}
\label{ueq1}
\det D^2 u=c \quad \mbox{in $Q$}, \quad\text{and} \quad u(x)=\frac{|x|^2}{2}~\text{on }\p Q
\end{equation}
for some constant $c\in (0,1)$. 

Notice that we are no longer assuming that $u\ge 0$ as in Section \ref{C2_sect}. 
The classification of global solutions relies on refined asymptotic analysis at infinity
of these solutions. Our arguments for a non-quadratic solution $u$ to (\ref{ueq1}) can be sketched as follows. 

First, we show in Lemma \ref{uPc-lem} that $u-P_c^-$ grows at most $|x|^{\beta_c^- +\sigma}$ at infinity for 
any $\sigma>0$.

Next, we establish a boundary Harnack principle at infinity for $u$. In Lemma \ref{l41} we show that after the affine transformation using $A_c^-$ that maps $Q$ to $Q_c^-$  and $P_c^-$ to $q$, 
the rescaled difference $(u-P_c^-)\circ A_c^-$ is asymptotically a nonnegative multiple of the positive, harmonic, homogenous of degree $\beta^-_c$ function $v_0$ defined in \eqref{v0fn}, that is
$$(u-P_c^-)\circ A_c^- = (a+ o(1)) v_0 \quad \text{at infinity on } Q_c^-,$$
for some constant $a$. 

This expansion allows us to apply the maximum principle in the unbounded domain $Q_c^-$. We construct two global solutions $\bar P_c$ and $\underbar{P}_c$ to (\ref{ueq1}) for which the corresponding constant $a$ changes sign. Using quadratic rescalings of these solutions together with $P_c^-$, we obtain a continuous family of solutions to (\ref{ueq1}) for which the constant $a$ ranges over the full $\mathbb R$.  The classification of global solutions
then follows by the maximum principle.

We first show that a solution $u$ to (\ref{ueq1}) which is different than $P_c^+$ must be close to $P_c^-$ at infinity.

\begin{lemma}
Assume that $u$ satisfies (\ref{ueq1}) and $u \ne P_c^+$. 
Then \begin{equation}
\label{D2w1}
\lim_{|x| \to \infty} D^2 u (x) = D^2 P_c^-.
\end{equation} 
\end{lemma}

\begin{proof}
First we show that 
\begin{equation}\label{400}
\mbox{$u(x) \to \infty$ as $|x| \to \infty$.}
\end{equation}
Indeed, we use $P_c^- -C(x_1+x_2)$ as a lower barrier for $u$ in $Q\cap B_1$ and deduce from the convexity of $u$ that
$$v:=u  + C(x_1+x_2) \ge 0 \quad \mbox{in $Q$}.$$
We consider the sections of $v$, $S_h:=\{x\in \overline{Q}: v(x) < h\}$ with $h$ large. Since $\det D^2 v=c$ we find
$|S_h| < C h$ for some large $C$ depending on $c$. 
On the other hand $S_h \subset \bar Q$ is a convex set which 
contains line segments of length $\frac 12 \sqrt h$ along $\p Q$ starting at the origin. 
In conclusion $S_h \subset B_{C \sqrt h}$ for some large $C$ 
which means that $v(x) \ge c_0 |x|^2$  for some $c_0(c)>0$ and for all large $|x|$ and our claim \eqref{400} is proved.  

As in Section \ref{Compactness}, we have from the convexity of $u$ that $u(x)\leq |x|^2$ in $Q$. We deduce from this and (\ref{400}) that the rescalings $$u_\lambda(x):=\lambda^{-2} u( \lambda x),$$ 
must converge uniformly on compact sets of $\ov Q$ along subsequences of $\lambda_k \to \infty$ to a solution 
$\bar u$ to  (\ref{ueq1}), and $\bar u \ge 0$ by (\ref{400}). If $\bar u \ne P_c^-$ then, by Lemma \ref{Pcpm2}, $\bar u _{12}(x) \to \sqrt {1-c}$ as $x \to 0$ and, after a quadratic rescaling by a factor we may assume 
$$|\bar u - P_c^+| \le \frac 12 \eps_0 \quad \mbox{in $Q \cap B_1$}$$
where $\eps_0=\eps_0(c)>0$ is the small constant  in Corollary \ref{Pc2a}.
This implies that
$$|u_{\lambda_k} - P_c^+| \le \eps_0 \quad \mbox{in $Q \cap B_1$},$$
for a sequence of $\lambda_k \to \infty$. By Corollary \ref{Pc2a} we obtain
$$|u_{\lambda_k} - P_c^+| \le C \eps_0 |x|^{2+\alpha} \quad \mbox{in $Q \cap B_1$},$$
which gives $u = P_c^+$, and we reach a contradiction. 

In conclusion $\bar u=P_c^-$ for any sequence of $\lambda \to \infty$. As in Lemma \ref{comp}, in $(B_2 \setminus B_{1/2}) \cap Q$ we have $\|u_{\lambda}-P_c^-\|_{C^2} \to 0$ which implies \eqref{D2w1}. 
\end{proof}

Next we establish the asymptotic behavior of solutions to (\ref{ueq1}) which have $P_c^-$ as a quadratic limit at infinity.  
\begin{lemma} 
\label{uPc-lem}
Assume that $u \ne P_c^+$ satisfies (\ref{ueq1}).
Then
for any $\sigma>0$, we have
\begin{equation}
\label{w_decay}
u(x) -P_c^-(x)= O(|x|^{\beta_c^- +\sigma})\quad \text{at infinity},
\end{equation}
and
\begin{equation}
\label{D2w2}
D^2 (u-P_c^-)(x)=O(|x|^{\beta_c^- +\sigma-2}) \quad \text{at infinity}.
\end{equation}
That is, for all  $|x|\geq R(\sigma, c)$, we have
$$|u(x) -P_c^-(x)|\leq  C(\sigma, c)(|x|^{\beta_c^- +\sigma})\quad \text{and }|D^2(u -P_c^-)(x)|\leq  C(\sigma, c)(|x|^{\beta_c^- +\sigma-2}).$$
\end{lemma}

\begin{proof}

We define $$w:=u-P_c^-.$$
Let $\eps_0=\eps_0(c,2)$ and $\rho= \rho(c, 2)$ be as in Proposition \ref{C2P-}.

First, by applying Proposition \ref{C2P-} in outgoing annuli towards infinity, we conclude that
\begin{equation}\label{w_decay1}w(x)=O(|x|^{2- \mu}) \quad \mbox{at infinity, with} \quad \quad\mu:= \frac{\log \frac{1}{2}}{\log \rho} .
\end{equation}
The proof of (\ref{w_decay1}) goes as follows.
First, by (\ref{D2w1}), we have
$$\lim_{|x| \to \infty} D^2 w(x)=0. $$
For each $\ep\in (0,\ep_0)$, using this and the Taylor formula,
we can find $R(\ep)>1$ such that
$$|w(z)|\leq \ep |z|^2=\ep |z|^{\beta_0}~\text{for all } z\in Q\setminus B_{R(\ep)}.$$
Here $\beta_0=2$ and hence $\rho =\rho(c,\beta_0)$. For all $z_0\in Q$ with $|z_0|= r\geq \frac{R(\ep)}{\rho}$, we apply
Proposition \ref{C2P-} to the function
$\hat w(z) =r^{-2} w(rz)$ with $$|\hat w(z)|\leq \ep r^{\beta_0-2}|z|^{\beta_0}~\text{for all } |z|\geq \rho$$
to obtain $|\hat w(z_0/r)| \leq \frac 12 \ep r^{\beta_0-2}$, which implies that
$|w(z_0)|\leq  \frac{\ep}{2}|z_0|^{\beta_0}$. Therefore, we have
$$|w(z)|\leq \frac{\ep}{2} |z|^{\beta_0}~\text{for all } z\in Q\setminus B_{\frac{R(\ep)}{\rho}}.$$
By induction, we obtain
$$|w(z)|\leq \frac{\ep}{2^k} |z|^{\beta_0}~\text{for all } z\in Q\setminus B_{\frac{R(\ep)}{\rho^k}}.$$
Then, for $|z|$ sufficiently large, we have
\begin{equation}
\label{w_decay2}
|w(z)| \leq 2 [R(\ep)]^{\mu} |z|^{\beta_0-\mu} = O(|z|^{\beta_0-\mu})=o(|z|^{2- \frac 12 \mu})
\end{equation}
from which (\ref{w_decay1}) easily follows.

Next, we show that the exponent $\beta:=2- \frac 12 \mu$ in (\ref{w_decay2}) can be lowered successively to become as close as we want to $\beta_c^- \in (1,2)$. 
 Indeed, if $\beta\leq \beta_c^-$ then we are done. Otherwise, the same rescaling argument as above shows that
 $$w(z) = O(|z|^{\beta-\mu_1})=o(|z|^{\beta- \frac 12 \mu_1})\quad\text{where }
 \mu_1:= \frac{\log \frac{1}{2}}{\log \rho(c,\beta)}.$$
Note that, if $\beta$ is bounded away from $\beta_c^-$ then $\rho(c,\beta)$ is also bounded away from $0$ by Remark \ref{r33}. Thus we can repeat the above argument and can replace $\beta$ by $\beta_c^-+\sigma$ for any 
$\sigma>0$, after a finite number of steps. In conclusion, we have
$w=O(|x|^{\beta_c^-+\sigma})$ at infinity from which is exactly (\ref{w_decay}).

Finally, we note that (\ref{D2w2}) is a consequence of  (\ref{w_decay}) and Schauder estimates (see \cite{GT}) applied to the equation
$$0= \det D^2 u-\det D^2 P_c^-= \text{trace} (A D^2 w)$$
where
$$A= \int_0^1 \text{cof} (D^2 P_c^- + t (D^2 u-D^2 P_c^-)) dt.$$
Here we use $\text{cof} (M)$ to denote the cofactor matrix of $M$.  Notice that by \eqref{uC2}-\eqref{uC3}, the coefficient matrix $A$ is uniformly elliptic and its first derivatives are bounded by $C|x|^{-1}$ at infinity.

\end{proof}

Before proceeding further, we recall the notation in Section \ref{Qcpm} and  Section  \ref{Harmonic_sect}. Let
 
$$A_c^{-}=
\begin{pmatrix} 1 &  \frac{\sqrt{1-c}}{\sqrt{c}}\\ 0 &\frac{1}{\sqrt{c}} \end{pmatrix},~
\quad \quad Q_c^{-}= (A_c^{-})^{-1}Q,$$
$$ v_0(r,\theta)= r^{\beta_c^-}\sin (\beta_c^- \theta), \quad \quad \beta_c^-\in (1,2).$$
We recall that $v_0$ is the positive, homogenous of degree $\beta^-_c \in (1,2)$ harmonic function in $Q_c^-$.

The next lemma establishes a boundary Harnack principle at infinity for non-quadratic solutions to (\ref{ueq})-(\ref{ubdr}). The precise statement is as follows.
\begin{lemma}\label{l41}
Assume that $u\ne P_c^+$ satisfies (\ref{ueq1}). 
Then
 \begin{equation}
 \label{HP}
 u\circ A_c^-= q+ (a + o(1)) v_0\quad \text{at infinity on } Q_c^-
 \end{equation}
for some constant $a$. 
\end{lemma}

\begin{proof}[Proof of Lemma \ref{l41}] 
We recall from Section \ref{Qcpm} that
$$P_c^- \circ A_c^- = q\quad \text{and }u_c^-:= u\circ A_c^-.$$
To simplify notation, let us denote
$$w:= u_c^--q= u\circ A_c^- -q.$$
We need to show that $w$ satisfies
 \begin{equation}
 \label{HPw}
 w= (a + o(1)) v_0\quad \text{at infinity on } Q_c^-
 \end{equation}
for some constant $a$. 

We start with the fact 
that $\det D^2 u_c^-=\det D^2 q=1$ in $Q_c^-$ and moreover, 
$w= u_c^--q$ solves a linearized equation 
$$a_{ij}w_{ij}=0 \quad \mbox{in $Q_c^-$},\quad \text{with } w=0\quad \text{on }\p Q_c^-.$$
Furthermore, by Lemma \ref{uPc-lem}, we have for any $\sigma>0$,  
$$|w(x)| \le C(\sigma, c) |x|^{\beta_c^{-} + \sigma}\quad \text{and } |a_{ij}(x)-\delta_{ij}| +  \, |D^2 w(x)| \quad  \le C(\sigma, c) |x|^{\beta_c^{-} + \sigma -2}\quad\text{at infinity on } Q_c^-.$$ 
At infinity, we have
$$\Delta w(x) = (\delta_{ij}- a_{ij}(x)) w_{ij}(x)=O (|x|^{2(\beta_c^{-} + \sigma -2)})$$
By choosing $\sigma\in (0, (2-\beta_c^-)/3]$, we find 
 $$\triangle w= f, \quad \text{with}\quad |f(x)| = O (|x|^{\beta_c^--\sigma-2 })\quad\text{at infinity on } Q_c^-.$$

 We can find (see Lemma \ref{v1_lem}) an explicit homogenous of degree $\beta_c^- - \sigma$ function $v_1 \ge 0$ on $Q_c^-$ which vanishes on the boundary of $Q_c^-$, such that
 $$\triangle  v_1(x) \le -|x|^{\beta_c^--\sigma-2 }\quad\text{on } Q_c^-.$$  
 This means that we can solve by Perron's method
 \begin{equation*}
\left\{
 \begin{alignedat}{2}
   \Delta v ~& = f~&&\text{in} ~  Q_c^-, \\\
v &= 0~&&\text{on}~ \p Q_c^-
 \end{alignedat} 
  \right.
  \end{equation*}
 for some function $v$ such that $-C v_1 \le v \le C v_1$. 
  It follows that $$w(x)-v(x)=O(|x|^{\beta_c^-+\sigma}) \quad\text{at infinity on } Q_c^-$$ is harmonic in $Q_c^-$ and vanishes on the boundary $\p Q_c^-$, thus
  \begin{equation}\label{wvL} w-v= a v_0,
  \end{equation}
 for some constant $a$. This can be easily seen using a conformal transformation mapping $Q_c^-$ to the upper half-plane $\H$ and arguing as in the proof of Lemma \ref{appro_lem2}.

Now on $\p B_1\cap Q_c^-$ we know that $v_0$ and $v_1$ are comparable. Recalling the homogeneities of $v_1$ and $v_0$ and 
using $|v|\leq C v_1$, we have 
$$v= o(1) v_0 \quad \text{at infinity on } Q_c^-.$$
Combining this with (\ref{wvL}), we conclude that $w= (a + o(1)) v_0$ at infinity on $Q_c^-.$

\end{proof}

\begin{corollary}
\label{MP_lem}
Assume that $u,\tilde u \ne P_c^+$ satisfy (\ref{ueq1}), and let $a$ and $\tilde a$ denote their corresponding constants  in the expansion \eqref{HP}. If $a< \tilde a$ then $u < \tilde u$ in $Q$.
\end{corollary}

Indeed, $a< \tilde a$ in the expansion  \eqref{HP} implies that $u< \tilde u$ on $A_c^-(\p B_r) \cap Q$ for all large $r$'s. Since $u= \tilde u$ on $\p Q$ and they both satisfy \eqref{ueq1}, we can apply the maximum principle in $A_c^-(Q_c^- \cap B_r)$ and conclude that $u< \tilde u$ in this set.

In the following lemma, we construct two particular solutions to \eqref{ueq1} that are not quadratic.
\begin{lemma}
\label{barPc}
There are two solutions $\bar P_c$, $\underbar P_c$ to \eqref{ueq1} so that
$$\underbar P_c < P_c^-<\bar P_c< P_c^+ \quad \mbox{in} \quad Q, \quad \mbox{and} \quad \bar P_c(1,1)=1, \quad \underbar P_c(1,1)=0.$$
At the origin $\bar P_c$ is pointwise $C^{2,\alpha}$ for some $\alpha=\alpha(c)\in (0, 1)$ 
and $\underbar P_c$ has a conical singularity. Moreover, their corresponding constants in the expansion at infinity \eqref{HP}  satisfy $\bar a>0$ and $\underbar a<0$.  At infinity, we have
$$\bar P_c(x) -P_c^-(x)= O(|x|^{\beta_c^- +\sigma})\quad \text{and } \underbar P_c(x) -P_c^-(x) = O(|x|^{\beta_c^- +\sigma}) \quad \text{for any } \sigma>0.$$
\end{lemma}

\begin{proof} We first construct $\bar P_c$.

For each $R>0$, we solve the Dirichlet problem on $Q \cap B_R$ 
\begin{equation}\label{PR}
\left\{
 \begin{alignedat}{2}
   \det D^2 P_R ~& = c~&&\text{in} ~  Q \cap B_R, \\\
P_R &= P_c^- + t_R \, \, x_1 x_2~&&\text{on}~ \p (Q \cap B_R),
 \end{alignedat} 
  \right.
  \end{equation}
 where $t_R \in (0, 2 \sqrt{1-c})$ is chosen such that the solution $P_R$ takes value 1 at $(1,1)$, that is, $$P_R(1,1)=1.$$ The existence of $t_R \in (0, 2 \sqrt{1-c})$ follows by continuity. In fact, when $t_R=0$, we have $P_R=P_c^-$ with $P_c^-(1,1)= 1-\sqrt{1-c}$, and when $t_R=2\sqrt{1-c}$, we have $P_R=P_c^+$ with
 $P_c^+ (1,1)= 1+\sqrt{1-c}$. 
 
 From $t_R \in (0, 2 \sqrt{1-c})$, we have $P_c^-\leq P_c^- + t_R x_1 x_2 \leq P_c^+$ on  $\p (Q \cap B_R)$. Thus, by the comparison principle for the Monge-Amp\`ere equation, we have
 $$P_c^-\leq P_R\leq P_c^+\quad \text{in} ~  Q \cap B_R,$$
hence the $P_R$'s are locally bounded independent on $R$.

We let $R \to \infty$ and, by compactness extract a convergence subsequence of $P_R$ to $\bar P_c$ satisfying (\ref{ueq1}) and 
$\bar P_c(1,1)=1$. Moreover, by the inequalities above we have $P_c^-<\bar P_c < P_c^+$ in $Q$.

Since $\bar P_c > P_c^-$, we obtain from Corollary \ref{MP_lem} that $\bar a \ge 0$. We claim that $\bar a$ cannot be $0$. 

Otherwise, let $u_c^-:=\bar P_c \circ A_c^-$ denote the affine deformation of $\bar P_c$ in the angle $Q_c^-$, and we have
\begin{equation}\label{401}
u_c^- = q+ o(1) v_0 \quad \text{at infinity on } Q_c^-.
\end{equation}
Thus, for each $\ep>0$, there is $R=R(\ep)$ large such that
 \begin{equation}
 \label{ucqv}
 u_c^-(x)\leq q(x) + \ep v_0 (x)\quad\text{for all } |x|\geq R.
 \end{equation}
 Since $\det D^2 u_c^-=1$ in $Q_c^-$, we have
 $$\Delta (q+\ep v_0)=2\leq \Delta u_c^-\quad \text{in } Q_c^-\cap B_{R} $$
 while  from (\ref{ucqv})
 $$u_c^-\leq q + \ep v_0\quad \text{on } \p (Q_c^-\cap B_{R}).$$
By the comparison principle, we have
 $u_c^-\leq q+\ep v_0 \quad \text{in } Q_c^-\cap B_{R},$
  hence, together with (\ref{ucqv}), we obtain
  $$ u_c^-\leq q+\ep v_0 \quad \text{in } Q_c^-.$$
By letting $\eps \to 0 $, we obtain
$u_c^-\leq q$ in $Q_c^-$.
Transforming back this inequality to the first quadrant $Q$, we find that $$\bar P_c \le P_c^-\quad\text{in } Q,$$
and we reached a contradiction.

Next we discuss the construction of $\underbar P_c$. For this we solve \eqref{PR} for each $R>2$ with $t_R \in (-1, 0) $ to obtain the solution $P_R$ so that $P_R(1,1)=0$. The existence of such a $t_R$ follows by continuity as above. In fact, when $t_R=-1$, we have the solution $P_R$ of (\ref{PR}) with $P_R(0)=0$ and $P_R(\frac{R}{\sqrt{2}}, \frac{R}{\sqrt{2}})=-\frac{\sqrt{1-c}}{2}R^2<0$, hence $P_R(1, 1)<0$ by convexity.

By symmetry we have $P_R \ge 0$ in $Q \cap B_R \cap \{ x_1+x_2 \ge 2\}$ which implies $P_R >-C$ in $Q$ for some $C$ universal.  From $t_R<0$, we have $P_R\leq P_c^-$ on  $\p (Q \cap B_R)$. Thus, by the comparison principle for the Monge-Amp\`ere equation, we have
 $$P_R\leq P_c^-\quad \text{in} ~  Q \cap B_R.$$
As above, we obtain by compactness the existence of $\underbar P_c$  satisfying (\ref{ueq1}) and 
$\underbar P_c(1,1)=0$. Also $\underbar P_c< P_c^-$ in $Q$ which gives $\underbar a \le 0$ in view of Corollary \ref{MP_lem}. We claim that $\underbar a$ cannot be $0$. 

Assume by contradiction that $\underbar a =0$. 
Denote as above $u_c^-:=\underbar P_c \circ A_c^-$ and \eqref{401} remains valid. 
Since $2\beta^-_c-2\in (0, \beta_c^-)$, by Lemma \ref{v1_lem}, there is a homogenous of degree $2\beta^-_c-2$ function $v_1\ge 0$  on $Q_c^-$ which vanishes on $\p Q_c^-$ and 
\begin{equation}
\label{v1b}
 \triangle v_1(x) \le -  |x|^{2\beta_c^--4}\quad \text{on } Q_c^-.
\end{equation}
Define
$$q_\eps^1:=q- \eps  p \cdot x, \quad \text{and } q_\eps^2:= q - \eps ( 1+ v_0 + v_1).$$
The linear function $p \cdot x$ is chosen such that $p \in Q_c^-$ and $q_\eps^1<q_\eps^2$ in $ (B_1 \setminus B_{1/2}) \cap Q_c^-$.
We show that
$$q_\eps:= \max\{ \chi_{B_1} \, q_\eps^1,\quad  q_\eps^2 \},$$
is a lower barrier for $u_c^-$ in $Q_c^- \cap B_R$ for some large $R$.  Clearly $q_\eps=q_\eps^1$ in a neighborhood of $0$ and $q_\eps=q_\eps^2$ outside $B_{1/2}$ hence $q_\eps \le u_c^-$ on $\p (Q_c^- \cap B_R)$ for $R$ large. For the interior inequalities in $Q_c^-\cap B_R$, we have $$\det D^2 q^1_\eps =1= \det D^2 u_c^-,$$ and outside a neighborhood of the origin, we have 
$$\det D^2 q^2_\eps=1 - \eps \triangle( v_1 + v_0) + O\left(\eps^2 |x|^{2(\beta_c^--2)}\right) >1,$$
for $\eps$ sufficiently small.   Here we used (\ref{v1b}).

In conclusion $q_\eps \le u_c^-$ in $Q_c^-$ and by letting $\eps \to 0 $, we obtain
$q \le u_c^-$, which gives $P_c^-\le \underbar P_c$
and we reached a contradiction.

Now, we establish the asymptotic behaviors of  $\underbar P_c$ and $\bar P_c$ at the origin and  infinity.  

Since $\bar P_c\geq 0$ is not quadratic, by Lemma \ref{Pcpm2}, we have $\lim_{x \to 0}\bar P_{c,12}(x)=\sqrt{1-c}$. Then, from Corollary \ref{cor_pm} and Corollary \ref{Pc2a}, we obtain the following asymptotic expasion
$$\bar P_c(x)= P_c^+ (x) + O(|x|^{2+\alpha}) \quad\text{near the origin}$$ for some $\alpha=\alpha(c)\in (0, 1)$. Hence,
$\bar P_c$ is pointwise $C^{2,\alpha}$ at the origin.

On the other hand, we note that  $\underbar P_c$ has a conical singularity at the origin, that is $\|D^2 \underbar P_c(x)\|\rightarrow 0$ as $x\rightarrow 0$. Indeed, suppose otherwise 
then the tangent plane of $\underbar P_c$ at the origin 
coincides with the tangent plane of $\frac{|x|^2}{2}$, hence $\underbar P_c \ge 0$ in $Q$. This is a contradiction because from $\underbar P_c(0)=\underbar P_c(1,1)=0$, we have
from the strict convexity of $\underbar P_c$ that $\underbar P_c(\frac{1}{2},\frac{1}{2})<0$.

Finally, since $\underbar P_c<\bar P_c<P_c^+$, by (\ref{w_decay}) of Lemma \ref{D2w1}, we have the asymptotic expansions for $\underbar P_c$ and $\bar P_c$
at infinity as stated in the lemma.
\end{proof}

We are now ready to state the main classification result of this section from which Theorem \ref{T2} easily follows.

\begin{propo}\label{Cl} 
Assume that $u$ satisfies (\ref{ueq1}). Then either $u= P_c^\pm$ or
$$u(x)= \lambda^{2}\, \, \bar P_c \left ( \frac x \lambda\right), \quad \mbox{or} \quad u(x)= \lambda^{2}\, \, \underbar P_c \left ( \frac x \lambda\right), $$ for some $\lambda \in (0, \infty)$. Here, $\bar P_c$, $\underbar P_c$ are two solutions to \eqref{ueq1} constructed in Lemma \ref{barPc}.
\end{propo}

\begin{proof}
Assume $u \ne P_c^+$, and let $a$ denote the constant of a solution $u$  
in the expansion \eqref{HP}. Then a quadratic rescaling of factor $\lambda$ 
of $u$ (that is, one of the form $\lambda^2 u(\frac{x}{\lambda})$) has constant $a \lambda^{2-\beta_c^-}$. 
By Lemma \ref{barPc}, $P_c^-$ and the two families of rescalings above 
generate an increasing continuous family of solutions indexed by constants $a$ in the expansion \eqref{HP}, with $a$ ranging over all $\mathbb R$. 
Now the classification result follows by the maximum principle in Corollary \ref{MP_lem}.

 \end{proof}
\begin{proof}[Proof of Theorem \ref{T2}] Combing Lemma \ref{uxx_lem}, Lemma \ref{barPc} and Proposition \ref{Cl}, we obtain the conclusions of Theorem \ref{T2}.
\end{proof}
 
 For completeness, we indicate a construction of $v_1$ alluded to in the proof of Lemma \ref{l41}.

\begin{lemma} 
\label{v1_lem}
Let $\beta\in [0,\beta_c^-)$.
There exists an 
explicit homogenous of degree $\beta$ function $v_1 \ge 0$ on $Q_c^-$ which vanishes on the boundary of $Q_c^-$, such that
 $\triangle  v_1(x) \le -|x|^{\beta-2 }\quad \text{on } Q_c^-.$ 
\end{lemma}
\begin{proof}
The opening angle of $Q_c^-$ is $\alpha_c^-= \frac{\pi}{\beta_c^-}$. We look for $v$ of the following form in polar coordinates
$$v(r, \theta)= r^{\beta}\varphi (\theta), 0\leq \theta\leq \alpha_c^-$$
where $\varphi(0)=\varphi(\alpha_c^-)=0$ and $\varphi(\theta)\geq 0$ for $0\leq \theta\leq \alpha_c^-$. 

Compute
$$\Delta v= r^{\beta-2}[\beta^2 \varphi(\theta) + \varphi^{''} (\theta)].$$
The problem reduces to finding $\varphi$ such that
$\beta^2 \varphi(\theta) + \varphi^{''}(\theta) < 0$ on $[0,\alpha_c^-],$ and then choosing $v_1=A v$ for some large constant $A$.

We can choose $\varphi$ of the form
$$\varphi(\theta) = \sin (\beta_c^- \theta)+ \delta \, \, \theta (\alpha_c^- -\theta)$$
with $\delta$ small. Indeed, for $\delta>0$ small, we have on $[0,\alpha_c^-]$
$$\beta^2 \varphi(\theta) + \varphi^{''}(\theta) = - ((\beta_c^-)^2-\beta^2) \sin  (\beta_c^- \theta)+\delta  \, \left [ \beta^2 \theta(\alpha_c^- -\theta)-2\right ] \, < 0.
$$

\end{proof}

\section{Proof of the global $C^{2,\alpha}$ estimates}
\label{T1_pf}
In this section, we prove Theorem \ref{T1} and its extension by using the results established in Proposition \ref{C2P+} and Proposition \ref{Cl}. 
\begin{proof}[Proof of Theorem \ref{T1}]
Let 
$u, \underbar u, f,\Om,\varphi,\beta$ be as in the statement of Theorem \ref{T1}. We proceed by showing first that $u$ is pointwise $C^{2,\alpha}$ at each vertex of $\Omega$, and then it is  
$C^{2,\alpha}$ in a neighborhood of each vertex,
and finally, u is globally  $C^{2,\alpha}$ in $\ov\Om$.
\vglue 2mm
\noindent
{\it Step 1: $u$ is pointwise $C^{2,\alpha}$ at each vertex.}
Consider a vertex of $\Omega$, which we can assume to be the origin $0$. 

We show that $u$ is pointwise $C^{2,\alpha}$ at $0$. 
After subtracting a linear function and after performing an affine transformation, we can assume: 

(1) the local geometry of $\Omega$ at $0$ is that of the first quadrant, 
$$
\Omega\cap B_\rho= Q\cap B_\rho \quad \text{for some } \rho \in (0,1).
$$

(2) $$\underbar u(0)=0, \quad  \nabla \underbar u(0)=0, \quad  \underbar u_{11}(0)=\underbar u_{22}(0)=1.$$

This implies that $u \ge \underbar u \ge 0$ and
$$\det D^2 u=f   \quad \mbox{in} \quad Q \cap B_\rho, \quad \quad u=\varphi \quad \mbox{on} \quad \p Q \cap B_\rho  $$
with 
$$|f(x)-f(0)|\le C|x|^{\beta} \quad \mbox{in} \quad Q \cap B_\rho, \quad |\varphi(x)-q(x)| \le C|x|^{2+\beta} \quad \mbox{on} \quad \p Q \cap B_\rho,$$
for some $C>0$ depending on $\|f\|_{C^\beta(\ov\Om)}$ and $\|\varphi\|_{C^{2,\beta}(\p\Om)}$. 

Define
$$c:=f(0),$$
and using that $\underbar u$ is a strict subsolution we have $c <1$ since
$$c=f(0) < \underbar f(0):=\det D^2 \underbar u(0) \le 1.$$

We claim that there exists $r$ small depending on the data above and the $C^2$ norm of $\underbar u$ such that the rescalings
$$u_r(x):= \frac {1}{r^2} \, u(rx), \quad f_r(x):=f(rx), \quad \varphi_r(x):= \frac {1}{r^2}   \, \varphi(rx),$$
satisfy the hypotheses of Corollary \ref{Pc2a}.
We can always choose $\alpha \le \beta$ if necessary in Corollary \ref{Pc2a}, so the only part that needs to be checked is 
\begin{equation}\label{501}
 |u_r(x)-P_c^+(x)| \le \ep_0 |x|^2 \quad \mbox{ in} \quad Q \cap B_1.
\end{equation}
 This follows by compactness. Indeed,  we have $$u \, \ge \, \underbar u = \frac 12 x^T D^2 \underbar u(0) \, x + o(|x|^2),$$
and any blow-up limit $\bar u$ of a sequence of $u_r$'s must be one of the global solutions characterized in Proposition \ref{Cl}. 
Since $\bar u$ is above the quadratic tangent polynomial of $\underbar u$ at the origin, which in turn separates quadratically above $P_c^-$ we find $\bar u=P_c^+$, which proves our claim.
\vglue 2mm
\noindent
{\it Step 2: $u$ is $C^{2,\alpha}$ in a neighborhood of each vertex.} Now it is standard to extend the pointwise $C^{2,\alpha}$ estimate from one vertex to $C^{2,\alpha}$ estimates in a neighborhood of that vertex. For this we use the $C^{2,\alpha}$ estimates at the boundary for the Monge-Amp\`ere equation (see \cite[Theorem 1.1]{S1}). 

Assume that we are in the setting of {\it Step 1}. 
Notice that as Section \ref{Compactness} we have that $u$ separates quadratically from its tangent plane at the boundary points on $ \p Q$ in annular domains 
$Q \cap (B_{4r} \setminus B_r)$ for all $r>0$ small. We can apply the results in \cite[Theorem 1.1]{S1} and conclude that
$$\|u_r-P_c^+\|_{C^{2,\alpha}} \le C r^\alpha \quad \mbox{in} \quad Q \cap (B_{3r} \setminus B_{2r}),$$
for all $r$ small. This implies that $u$ is $C^{2,\alpha}$ in a neighborhood of the origin.
\vglue 2mm
\noindent
{\it Step 3: Conclusion.} Having proved that $u$ is $C^{2,\alpha}$ in a neighborhood of each vertex of $\Omega$, we can combine these with 
$C^{2,\alpha}$ estimates  for the Monge-Amp\`ere equation at the boundary (see \cite[Theorem 1.1]{S1}) and  in the interior (see \cite{C}) to conclude that $u\in C^{2,\alpha}(\overline{\Omega})$.
\end{proof}

Next we give a version of Theorem \ref{T1} in which the hypothesis that $\underbar u$ is a strict subsolution is removed and we 
list all possible scenarios for the regularity of $u$ at the origin. For simplicity we assume that 
$$\Omega:=Q \cap B_1.$$

\begin{thm}\label{T3}
Assume that $u$ is a convex function that satisfies 
$$
\left\{
 \begin{alignedat}{2}
   \det D^2 u ~& = f~&&\text{in} ~  \Omega, \\\
u &= \varphi~&&\text{on}~ \p Q
 \end{alignedat} 
  \right.
  $$
 where for some $\beta\in (0, 1)$, $$f \in C^\beta(\overline{\Om}), \quad f>0, \quad{ and} \quad \varphi \in C^{2,\beta} (\p Q\cap B_1).$$ 
\begin{myindentpar}{1cm}
(i) If $f(0)< \varphi_{11}(0) \varphi_{22}(0)$ then either $u$ is $C^{2,\alpha}$ in a neighborhood of the origin for some $\alpha>0$ or $u$ has a conical singularity at $0$. \\
(ii) If $f(0) =  \varphi_{11}(0) \varphi_{22}(0)$ then either $u$ is $C^{2}$ in a neighborhood of the origin or $u$ has a conical singularity at $0$.\\
(iii) If $f(0) > \varphi_{11}(0) \varphi_{22}(0)$ then $u$ has a conical singularity at $0$.
\end{myindentpar}
\end{thm}

\begin{proof} Assume that $\varphi(0)=0$, $\nabla \varphi(0)=0$. If $u$ has a conical singularity at $0$ then we are done. Now, suppose that
$u$ does not have a conical singularity at $0$. Then its tangent plane at the origin 
coincides with the tangent plane of $\varphi$, hence $u \ge 0$ in $\Omega$.

The proof of (i) is essentially given in that of Theorem \ref{T1} above. The only difference is that now the blow-up limit $\bar u \ge 0$ 
can also be $P_c^-$ or a quadratic rescaling of $\bar P_c$. 
In the second case, after a rescaling by a large factor we end up again in the situation \eqref{501}. 
On the other hand, if $\bar u=P_c^-$ for any blowup limit of the $u_r$'s, then we are in the setting of Lemma \ref{Pc-cor}.
Now we obtain that $u$ is $C^{2,\alpha}$ at the origin with $P_c^-$ as its quadratic tangent polynomial at the 
origin.

The proof of (ii) corresponds to the case $c=1$ of Theorem \ref{T2}.
 Then the blowup limit $\bar u$ is unique $\bar u=q$ which gives that $u$ is pointwise $C^2$ at the origin. We can extend this estimate in a neighborhood of $0$ as in the proof of  Theorem \ref{T1} above.

The case (iii) corresponds to $c>1$ and it is obvious by Theorem \ref{T2}.
\end{proof}

\begin{rem} The $C^{2,\alpha}$ norm of $u$ cannot be easily quantified in the case (i) of Theorem \ref{T3} above. 
This is because by Proposition \ref{Cl} the quadratic polynomial $P_c^-$ is unstable for the $C^2$ norm: any small postive perturbation on $\p B_1 \cap Q$ 
produces a jump of order 1 for $D^2 u(0)$ while a small negative perturbation produces a conical singularity at the origin, i.e., $\|
D^2u(x)\| \to \infty$ as $x \to 0$. On the other hand, in Theorem \ref{T1} the existence of a global strict subsolution $\underbar u \in C^2$ prevents $D^2u$ being close to $D^2P_c^-$ near the origin.   
\end{rem}

We finally mention that our results in Theorem \ref{T3} are sharp in the sense that $ u \notin C^{2,\alpha}(0)$ in the case (ii). Indeed, if $c=1$ and consider a solution to 
$$\det D^2 u=1 \quad \mbox{in} \quad Q \cap B_1, \quad \quad u=q \quad  \mbox{on} \quad \p Q,$$
with $$u \ge q + \eps x_1 x_2 \quad \quad \mbox{on} \quad \p B_1 \cap Q.$$ Then $u \ge q$ by the maximum principle and, as 
shown above $q$ is the 
tangent quadratic polynomial of $u$ at the origin. We claim that
\begin{equation}
\label{ulog}
u \ge q + (\eps-C \eps^2) x_1 x_2 \quad \quad \mbox{on} \quad \p B_{1/2} \cap Q,
\end{equation}
which after iteration implies that $$u \ge q + \min \left\{ \eps/2, c' |\log |x||^{-1} \right \}  \, \, x_1x_2,$$ for some small $c'>0$. This shows that 
$ u \notin C^{2,\alpha}(0)$ for any $\alpha>0$. 

The claim (\ref{ulog}) follows from the maximum principle by checking that
$$q+ \eps x_1 x_2+ \eps^2 v$$
is a lower barrier for $u$, where $v$ is a $C^2$ function that satisfies
$$\triangle v \ge  2, \quad \|D^2v\| \le C \quad \mbox{in} \quad Q \cap B_1,$$
and
$$  v \le 0 \quad \mbox{on} \quad \p (Q \cap B_1), \quad \quad v=0 \quad \mbox{on} \quad \p Q \cap (B_{3/4} \setminus B_{1/4}).$$
\vglue 0.5cm
{\bf Acknowledgements.} The authors would like to thank the referee for helpful comments on the manuscript.

\end{document}